\documentclass{amsart}

\usepackage{amsmath}
\usepackage{amssymb}
\usepackage[]{amsrefs}
\usepackage{xpatch}
\usepackage{kantlipsum} 
\calclayout

\xpatchcmd{\proof}{\itshape}{\prooflabelfont}{}{}
\newcommand{\prooflabelfont}{\bfseries}

\DefineSimpleKey{bib}{primaryclass}{}
\DefineSimpleKey{bib}{archiveprefix}{}

\BibSpec{arXiv}{%
  +{}{\PrintAuthors}{author}
  +{,}{ \textit}{title}
  +{}{ \parenthesize}{date}
  +{,}{ arXiv }{eprint}
  +{,}{ primary class }{primaryclass}
}

\usepackage{kantlipsum} 
\calclayout
\usepackage{extarrows}
\usepackage{amssymb}
\usepackage[utf8]{inputenc}
\newtheorem{theorem}{Theorem}[section]
\newtheorem{proposition}[theorem]{Proposition}
\newtheorem{lemma}[theorem]{Lemma}
\newtheorem{corollary}[theorem]{Corollary}
\theoremstyle{definition}

\theoremstyle{definition}
\newtheorem{definition}[theorem]{Definition}
\newtheorem{remark}[theorem]{Remark}

\usepackage{amsmath}

\numberwithin{equation}{section}
\usepackage{amsthm,amsmath,amssymb}
\usepackage[all,cmtip]{xy}
\usepackage{amsmath}
\usepackage{amssymb}
\usepackage{amsthm}
\usepackage[]{amsrefs}
\usepackage{hyperref}
\usepackage{xpatch}
\usepackage{amsfonts}
\usepackage{amssymb}
\usepackage[utf8]{inputenc}
\usepackage{amsthm}
\usepackage{stmaryrd}
\usepackage{csquotes}
\usepackage{extarrows}
\MakeOuterQuote{"}
\theoremstyle{definition}

\usepackage{enumitem}
\usepackage{tikz-cd}

\usepackage{blindtext}

\usepackage{url}

\DeclareMathOperator{\Gll}{g\ell\ell}

\DeclareMathOperator{\Gr}{gr}

\DeclareMathOperator{\Ord}{ord}

\begin{document}

\title[The index of a numerical semigroup ring]{The index of a numerical semigroup ring}

\author[Bartels]{Richard Bartels}

\author[Dajani]{Sarah Dajani}

\author[Koomson]{Gabriel Koomson}

\address{Department of Mathematics, Trinity College, 300 Summit St, Hartford, CT, 06106}

\email{richard.bartels@trincoll.edu}
\urladdr{https://sites.google.com/view/richard-bartels-math/home}
\email{sarah.dajani@trincoll.edu}
\email{gabriel.koomson@trincoll.edu}



\title{Reduction numbers for witnesses to the generalized Loewy length}
\begin{abstract}  
Let $(R,\mathfrak{m})$ be a one-dimensional Cohen-Macaulay local ring. In this paper, we find the reduction number $r_{z}(\mathfrak{m}^d)$ of $\mathfrak{m}^d$ with respect to a witness $z\in \mathfrak{m}^d \setminus \mathfrak{m}^{d+1}$ to the generalized Loewy length $\text{g}\ell\ell(R)$ for several infinite families of hypersurfaces $\left\{(R,\mathfrak{m}) \right\}$. For every principal reduction $w$ of $\mathfrak{m}^d$, we have $\text{g}\ell\ell(R) \leq d(r_{w}(\mathfrak{m}^d)+1)$. We give examples of families $\left\{ (R,\mathfrak{m}) \right\}$ such that $d(r_{z}(\mathfrak{m}^d)+1)-\text{g}\ell\ell(R)=0$ and $d(r_{z}(\mathfrak{m}^d)+1)-\text{g}\ell\ell(R)=1$. In particular, we prove that for every prime $p$, there are local hypersurfaces $R$ over $\mathbb{F}_p$ such that $d(r_{z}(\mathfrak{m}^d)+1)-\Gll(R)=0$ and $d(r_{z}(\mathfrak{m}^d)+1)-\Gll(R)=1.$ We also prove that the difference $d(r_{z}(\mathfrak{m}^d)+1)-\text{g}\ell\ell(R)$ can vary independently of $\text{g}\ell\ell(R)-e(R)$.
\end{abstract}
\maketitle
\large{
\begin{center}
\section{Introduction}
\end{center}

Throughout this paper, $(R,\mathfrak{m})$ is a one-dimensional Cohen-Macaulay local ring, $k$ is a field, and $e(R)$ denotes the Hilbert-Samuel multiplicity of $R$. For a nonzero element $z \in R$, we let $\text{ord}_{R}(z)$ denote the {\it{order}} of $z$ in $R$; i.e., $\text{ord}_{R}(z)=d$ if $z \in \mathfrak{m}^d \setminus \mathfrak{m}^{d+1}$. If $z \in R$ is an order $d$ element, then the {\it{initial form}} of $z$, denoted $z^*$, is the residue class of $z$ in the quotient $\dfrac{\mathfrak{m}^d}{\mathfrak{m}^{d+1}}$. 
\\\text{}

A {\it{reduction}} of an ideal $I \subset R$ is an ideal $J \subset I$ such that $I^{n+1}=JI^n$ for $n \gg 0$. The smallest non-negative integer $n$ such that $I^{n+1}=JI^n$ is called the {\it{reduction number of $I$ with respect to $J$}} and denoted $r_{J}(I)$. A reduction $J$ of $I$ is {\it{minimal}} if $J$ contains no other reduction of $I$. The {\it{reduction number of}} $I$, denoted $r(I)$, is defined as \\
\begin{equation*}
r(I):=\text{min}\{r_{J}(I)\,|\,J \,\, \text{is a minimal reduction of} \,\, I\}.
\end{equation*}
\text{}

An ideal with no reduction other than itself is called {\it{basic}} \cite{HS06}[Chapter 8]. Suppose $z \in \mathfrak{m}$ with $\text{ord}_{R}(z)=d \geq 1$. Then $z$ is a {\it{superficial element of order $d$}} if there is an integer $c$ such that \\
\begin{equation*}
(\mathfrak{m}^{n+d}:z) \,\cap \, \mathfrak{m}^c=\mathfrak{m}^n
\end{equation*} for all $n \geq c$ \cite{De Stefani16}[Definition 2.1].\\ \text{}

Let $(R,\mathfrak{m})$ be a one-dimensional Cohen-Macaulay local ring with associated graded ring $\text{gr}_{\mathfrak{m}}(R)=\bigoplus\limits_{i=0}^{\infty}\dfrac{\mathfrak{m}^i}{\mathfrak{m}^{i+1}}$. Let $d \geq 1$. In \cite{De Stefani16}[Remark 2.3], De Stefani notes that if $z \in \mathfrak{m}^d \setminus \mathfrak{m}^{d+1}$ is a nonzerodivisor, then the following three conditions are equivalent: \\\\
\begin{enumerate}
\item[$(1)$] $z$ is a superficial element (of order $d$). \\
\item[$(2)$] $z$ is a reduction of $\mathfrak{m}^d$. \\
\item[$(3)$] The initial form $z^*$ is a homogeneous parameter in $\text{gr}_{\mathfrak{m}}(R)$.
\end{enumerate}\text{}\\
\begin{definition}\cite{LW12}[p.195]  Let $(R,\mathfrak{m})$ be a one-dimensional Cohen-Macaulay local ring. The {\it{generalized Loewy length}} of $R$, denoted $\text{g}\ell\ell(R)$, is defined as follows.\\
\[
\text{g}\ell\ell(R):=\text{min}\{i \geq 1\,|\,\mathfrak{m}^i \subset zR \,\,\text{for some nonzerodivisor}\,\, z \in \mathfrak{m}\}.
\]\text{}\\
A nonzerodivisor $z \in \mathfrak{m}$ is called a {\it{witness}} to $\text{g}\ell\ell(R)$ if $\mathfrak{m}^g \subset zR$, where $g=\text{g}\ell\ell(R)$.
\end{definition}
\text{}\\ \text{}

Suppose $z \in \mathfrak{m}^d \setminus \mathfrak{m}^{d+1}$ is a reduction of $\mathfrak{m}^d$ and $n=r_{z}(\mathfrak{m}^d)$ is the reduction number of $\mathfrak{m}^d$ with respect to $z$. Then 
\begin{equation*}
\mathfrak{m}^{(n+1)d} \subset z\mathfrak{m}^{nd} \subset zR.
\end{equation*}\text{}\\ 
Therefore, the generalized Loewy length $\text{g}\ell\ell(R)$ is a lower bound for $d(n+1)$:\\
\begin{equation*}
\text{g}\ell\ell(R) \leq d(n+1).
\end{equation*}\text{}

When $R$ is Gorenstein, we have $\text{index}(R) \leq \text{g}\ell\ell(R)$, where $\text{index}(R)$ denotes Auslander's index \cite{LW12}[p.195]. Let $S=k[[x,y]]$. For a local hypersurface $R=S/fS$, where $f \in (x,y)S$, we have $\text{ord}_{S}(f)=e(R)=\text{index}(R)$ \cite{Ding92}[Theorem 3.3]. Therefore, \\
\begin{equation}\label{equation:1.1}
\text{ord}_{S}(f)=e(R) \leq \text{g}\ell\ell(R).
\end{equation}\text{}\\
 In \cite{Bartels24}[section 3], we used (\ref{equation:1.1}), the following version of \cite{Bartels24}[Corollary 2.5], and \cite{Bartels24}[Lemma 3.4] to find the generalized Loewy length of several infinite families of one-dimensional hypersurfaces $\left\{R_i \right\}$.\text{}\\
\begin{proposition}\cite[Corollary 2.5]{Bartels24}\label{prop:1.1}
Let $S=k[[x,y]]$ and $R=S/fS$, where $f \in (x,y)S$. Let $\mathfrak{m}=(x,y)R$. Suppose $z \in \mathfrak{m}^d \setminus \mathfrak{m}^{d+1}$ such that the initial form $z^*$ is $\Gr_{\mathfrak{m}}(R)$-regular. Then 
\begin{equation*}
e(R) \leq \Gll (R) \leq e(R)+d-1.
\end{equation*} \text{}\\
Let $e=e(R)$. If $\mathfrak{m}^{e+d-1}$ is not principal, then $\mathfrak{m}^{e+d-1} \subset zR$.
\end{proposition} \text{}
\begin{proposition}\cite[Lemma 3.4]{Bartels24}\label{prop:1.2} Let $R=k[[x,y]]/(f)$, where $e(R)=\Ord_{S}(f) \geq 2$. If $R$ has no nonzerodivisor of the form $\alpha x + \beta y$, where $\alpha, \beta \in k$, then \\
\[
e(R) < \Gll(R).
\]
\end{proposition}\text{}
\begin{remark}
Each hypersurface $(R, \mathfrak{m})$ in the aforementioned families $\left\{ R_{i} \right\}$ has a witness $z \in \mathfrak{m}^d \setminus \mathfrak{m}^{d+1}$ to $\text{g}\ell\ell(R)$ with initial form $z^*$ that is a nonzerodivisor in $\text{gr}_{\mathfrak{m}}(R)$. Therefore, $z$ generates a reduction of $\mathfrak{m}^d$. Moreover, the value of $\text{g}\ell\ell(R)-e(R)$ depends entirely on the order of $z$.\end{remark}
\text{}\\
\begin{proposition}\cite[Proposition 3.3]{Bartels24}\label{prop:1.3}
Let $S=k[[x,y]]$ and $R=k[[x,y]]/(f)$, where $f \in (x,y)S$. Let $\mathfrak{m}=(x,y)R$. Suppose $z \in \mathfrak{m}^{d} \setminus \mathfrak{m}^{d+1}$ such that the initial form $z^*$ is $\Gr_{\mathfrak{m}}(R)$-regular. If $z$ is a witness to $\Gll(R)$, then
\[
\Gll (R) = e(R) +d -1.
\]
\end{proposition}
\text{}

In section 2, we use our knowledge of $\text{g}\ell\ell(R)$ for these families $\{R_i\}$ and the inequality $\text{g}\ell\ell(R) \leq d(n+1)$ to determine the reduction number $n=r_{z}(\mathfrak{m}^d)$ of $\mathfrak{m}^d$ with respect to witnesses $z$ to $\text{g}\ell\ell(R)$. We prove that for every prime $p$, there are hypersurfaces $(R,\mathfrak{m})$ over $\mathbb{F}_p$ such that $d(n+1)-\Gll(R)=0$ and $d(n+1)-\Gll(R)=1$ (Theorem \ref{theorem:2.16}). We also show that the difference $d(n+1)-\text{g}\ell\ell(R)$ can vary independently of $\text{g}\ell\ell(R)-e(R)=d-1$. 

In section 3, we find a reduction $z \in \mathfrak{m}^2_n \setminus \mathfrak{m}^3_n$ of $\mathfrak{m}_n^2$ for an additional family of hypersurfaces $\left\{(R_n, \mathfrak{m}_n ) \right\}_{n=1}^{\infty}$. We use this reduction and Proposition \ref{prop:1.2} to find the generalized Loewy length for this family and prove that, for $n \geq 1$, the reduction $z \in \mathfrak{m}^2_n$ is a witness to $\text{g}\ell\ell(R_n)$ such that \\
\[
2(2^{n-1}+1)=\text{g}\ell\ell(R_n)=2(r_{z}(\mathfrak{m}_{n}^2)+1)
\]
and
\[
r_{z}(\mathfrak{m}_n^2)=2^{n-1}.
\]
\text{}\\
\begin{center}
\section{Reduction numbers for local hypersurfaces} 
\end{center}

In the following, $\mathbb{F}_p$ denotes the field with $p$ elements, where $p$ is prime. We prove that for every prime $p$, there are hypersurfaces $(R,\mathfrak{m})$ over $\mathbb{F}_p$ with order $d \geq 1$ witnesses $z$ to $\text{g}\ell\ell(R)$ such that $d(r_{z}(\mathfrak{m}^d)+1)-\text{g}\ell\ell(R)=0$ and $d(r_{z}(\mathfrak{m}^d)+1)-\text{g}\ell\ell(R)=1$.\text{}\\\\
\begin{lemma}\cite[section 2]{D'Anna01}\label{lemma:2.2.1}
Let $R$ be a Noetherian ring and $I \subset R$ an ideal that contains a nonzerodivisor. If $w_1, w_2 \in I$ are elements that generate principal reductions of $I$, then $r_{w_1}(I)=r_{w_2}(I)$.
\end{lemma} 
\text{}
\begin{lemma}\cite[Corollaries 8.3.6, 8.3.9]{HS06}\label{lemma:2.1.1}
Let $(R,\mathfrak{m})$ be a one-dimensional Noetherian local ring. Let $I \subset R$ be an $\mathfrak{m}$-primary ideal. Suppose $w \in I$ is a reduction of $I$. Then $w$ is a minimal reduction of $I$.
\end{lemma}
\text{}
\begin{lemma}\label{lemma:2.3}
Let $(R,\mathfrak{m})$ be a one-dimensional Cohen-Macaulay local ring with infinite residue field. Suppose $I \subset R$ is an $\mathfrak{m}$-primary ideal. Then every minimal reduction of $I$ is principal and $r(I)=r_{w}(I)$ for every principal reduction $w$ of $I$.
\end{lemma}
\text{}
\begin{proof}
Since $I$ is $\mathfrak{m}$-primary, $I$ contains a nonzerodivisor and has height one. By \cite{HS06}[Corollary 8.3.9], the analytic spread of $I$ equals one. Since $k$ is infinite, it follows from \cite{HS06}[Proposition 8.3.7] that every minimal reduction of $I$ is a principal reduction of $I$. Therefore, principal reductions and minimal reductions of $I$ are the same by Lemma \ref{lemma:2.1.1}, and the result follows from Lemma \ref{lemma:2.2.1}.
\end{proof}
\text{}
\begin{remark}
The hypersurfaces $(R,\mathfrak{m})$ in Proposition \ref{prop:2.3} satisfy $e(R)=\text{g}\ell\ell(R)=r_{z}(\mathfrak{m}) + 1$, where $z \in \mathfrak{m}\setminus \mathfrak{m}^2$ is a witness to $\text{g}\ell\ell(R)$ that generates a reduction of $\mathfrak{m}$. When the residue field $k$ is infinite, the reduction number $r(\mathfrak{m})$ equals $r_{w}(\mathfrak{m})$ for every principal reduction $w$ of $\mathfrak{m}$.
\end{remark}
\text{}\\
\begin{proposition}\label{prop:2.3}
Let $n \geq 1$ and $k$ a field such that $\text{char}\, k \neq 2$ and $\text{char}\,k \nmid 1 + (-2)^n$. Let 
\[
R=k[[x,y]]\bigg/\left(xy\left(x^n + y^n \right) \right)
\] and $\mathfrak{m}=(x,y)R$. Then $z=x+2y \in \mathfrak{m} \setminus \mathfrak{m}^2$ is a witness to $\Gll(R)$ that generates a minimal reduction of $\mathfrak{m}$. We have \\
\[
n+2=e(R)=\Gll(R)=r_{z}(\mathfrak{m}) +1
\] and 
\[
r_{z}(\mathfrak{m})=n+1.
\] \text{}\\ Therefore, $r_{w}(\mathfrak{m})=n+1$ for every principal reduction $w$ of $\mathfrak{m}$. If $k$ is infinite, then $r(\mathfrak{m})=n+1$.
\end{proposition}
\text{}\\
\begin{proof}
By \cite{Bartels24}[Proposition 3.7], we have 
\[
\mathfrak{m}^{n+2}=z\mathfrak{m}^{n+1}
\]
and $e(R)=\text{g}\ell\ell(R)=n+2$. Let $r=r_{z}(\mathfrak{m})$. Then $r \leq n+1$. If $r < n+1$, then $\mathfrak{m}^{r +1}=z\mathfrak{m}^{r} \subset zR$, so $\text{g}\ell\ell(R) \leq r +1<n+2=\text{g}\ell\ell(R)$, a contradiction. By Lemma \ref{lemma:2.2.1}, we have $r_{w}(\mathfrak{m})=n+1$ for every principal reduction $w$ of $\mathfrak{m}$. If $k$ is infinite, then $r_{z}(\mathfrak{m})=r(\mathfrak{m})$ by Lemma \ref{lemma:2.3}.  
\end{proof} \text{}\\
\begin{corollary}\label{corollary:2.6}
Let $k$ be a field of characteristic $p>2$ and let $n \geq 0$. Let \\
\[
R=k[[x,y]] \bigg/\left(xy\left(x^{p^n} + y^{p^n} \right) \right),
\] \text{}\\ $\mathfrak{m}=(x,y)R$, and $z=x+2y \in \mathfrak{m} \setminus \mathfrak{m}^2$. Then $z$ is a witness to $\Gll(R)$ and generates a reduction of $\mathfrak{m}$. We have \\
\[
p^n + 2 = e(R)=\Gll(R)=r_{z}(\mathfrak{m}) + 1
\] and 
\[
r_{z}(\mathfrak{m})=p^n + 1.
\] \text{}\\ Therefore, $r_{w}(\mathfrak{m})=p^n+1$ for every principal reduction $w$ of $\mathfrak{m}$. If $k$ is infinite, then $r(\mathfrak{m})=p^n + 1$.
\end{corollary}
\text{}\\
\begin{proof}
For a prime $p>2$, we have $p \nmid 1 + (-2)^{p^n}$.
\end{proof}
\text{}\\
\begin{remark}
For all hypersurfaces in the family $\left\{(R,\mathfrak{m}) \right\}$ in Proposition \ref{prop:2.6}, we have $
\text{g}\ell\ell(R)-e(R)=1$. In Propositions \ref{prop:2.8} and \ref{prop:2.9}, we determine the reduction number $r_{z}(\mathfrak{m}^2)$ of $\mathfrak{m}^2$ with respect to the witness $z=x^2+xy+y^2 \in R$ to $\text{g}\ell\ell(R)$ for two (possibly infinite) subsets of this family (see \cite{Bartels24}[Remark 3.13] and \cite{CDR04}[p.66]). For the first subset, we have \\ 
\[
2(r_{z}(\mathfrak{m}^2) +1)-\text{g}\ell\ell(R)=0.
\]
For the second, we have 
\[
2(r_{z}(\mathfrak{m}^2) +1)-\text{g}\ell\ell(R)=1.
\]
\end{remark} \text{}\\
\begin{proposition}\cite[Proposition 3.12]{Bartels24}\label{prop:2.6}
Let $m, n \geq 0$ and let $p>3$ be a prime such that $2$ is a primitive root modulo $p^2$. Let \\
\[
R=\mathbb{F}_2[[x,y]] \bigg/ \left(xy \left(x^{2^n p^m} + y^{2^n p^m} \right)\right)
\] \text{}\\ and $\mathfrak{m}=(x,y)R$.
Then \\
\[
\Gll(R)=e(R)+1=2^n p^m +3
\] \text{}\\ and $z=x^2+xy+y^2 \in R$ is a witness to $\Gll(R)$. If $m=1$, then we need only assume that $2$ is a primitive root modulo $p$.
\end{proposition} \text{}\\\\
\begin{proposition}\label{prop:2.8}
Let $p>3$ be a prime such that $3 \mid p-1$ and $2$ is a primitive root modulo $p$. Let \\
\[
R=\mathbb{F}_2[[x,y]] \bigg/ \left(xy \left(x^p + y^p \right) \right)
\] \text{}\\ and $\mathfrak{m}=(x,y)R$. Let $z=x^2 + xy + y^2 \in \mathfrak{m}^2 \setminus \mathfrak{m}^3$. Then $z$ is a witness to $\Gll(R)$ and generates a minimal reduction of $\mathfrak{m}^2$. We have \\
\[
p+3=\Gll(R)=2 \left(r_{z}(\mathfrak{m}^2) + 1 \right)
\] \text{}\\ and
\[
r_{z}(\mathfrak{m}^2) = \dfrac{p+1}{2}.
\] 
\text{}\\\\
Therefore, $r_{w}(\mathfrak{m}^2)=\dfrac{p+1}{2}$ for every principal reduction $w$ of $\mathfrak{m}^2$.
\end{proposition}
\text{}\\
\begin{proof}
By Proposition \ref{prop:2.6}, $z$ is a witness to $\text{g}\ell\ell(R)$ and $\text{g}\ell\ell(R)=p+3$. By the proof of Proposition \ref{prop:2.6}, the initial form $z^*$ is $\text{gr}_{\mathfrak{m}}(R)$-regular. Therefore, $z$ generates a reduction of $\mathfrak{m}^2$ and \\
\begin{equation*}\label{equation:1}
p+3=\text{g}\ell\ell(R) \leq 2 \left(r_{z}(\mathfrak{m}^2) + 1 \right).
\end{equation*}\text{}\\
We prove that $p+3=2\left(r_{z}(\mathfrak{m}^2) + 1 \right)$. To this end, we prove the inclusion $\mathfrak{m}^{p+3} \subseteq z \mathfrak{m}^{p+1}$. We have 
\[
\mathfrak{m}^{p+1}=\left(x^{p+1},\, x^p y, \, \dots, y^{p+1} \right)R.
\] \text{}\\ The $p+2$ generators $x^{p+1-i}y^i$ of $\mathfrak{m}^{p+1}$ give us the following $p+2$ generators $z_i$ of $z\mathfrak{m}^{p+1}$:\\
\begin{align*}
&z_1 = x^{p+3} + x^{p+2}y + x^{p+1}y^2&\\\\
&z_2 = x^{p+2}y + x^{p+1}y^2 + x^{p}y^3&\\\\
&\vdots& \\\\
&z_{p+2}=x^2 y^{p+1} +xy^{p+2} + y^{p+3}&
\end{align*} \text{}\\ and in general, 
\begin{align*}
&z_i=x^{p+4-i}y^{i-1} + x^{p+3-i}y^{i} + x^{p+2-i}y^{i+1}&
\end{align*}\text{}\\ 
for $1 \leq i \leq p+2$. Since $x^{p+1} y= xy^{p+1}$ in $R$, we have \\ 
\begin{equation}\label{equation:2.1}
x^{p+2}y=x^2y^{p+1}
\end{equation} and 
\begin{equation}\label{equation:2.2}
x^{p+1}y^2 = xy^{p+2}.
\end{equation} \text{}\\ Define the elements $w_1, \, w_2, \, \dots, \, w_{(p+2)/3} \in z\mathfrak{m}^{p+1}$ as follows: \\
\begin{align*}
&w_1 = z_2 + z_3 = x^{p+2}y + x^{p-1}y^4&\\\\
&w_2 = z_5 +z _6 = x^{p-1}y^4 + x^{p-4}y^{7}&\\\\ 
&\vdots& \\\\
&w_{(p-1)/3}=z_{p-2} + z_{p-1} = x^7 y^{p-4} + x^4 y^{p-1}&
\\\\
&w_{(p+2)/3}= z_{p+1} + z_{p+2} = x^4 y^{p-1} + x y^{p+2}&
\end{align*} \text{}\\ and in general, 
\begin{align*}
&w_i = z_{3i-1} + z_{3i}& 
\end{align*}
\text{}\\ for $1 \leq i \leq (p+2)/3.$
By equation \ref{equation:2.2}, we have \\
\begin{equation}\label{equation:2.4}
\sum\limits_{i=1}^{(p+2)/3} w_{i} = x^{p+2}y + xy^{p+2} = x^{p+2}y + x^{p+1}y^2 \in z \mathfrak{m}^{p+1}.
\end{equation} \text{}\\ By equation \ref{equation:2.4} and the definition of $z_1$, it follows that $x^{p+3} \in z \mathfrak{m}^{p+1}$. Likewise, define elements $v_1, v_2, \dots, v_{(p-1)/3}$ as follows: \\
\begin{align*}
&v_1 = z_3 + z_4 = x^{p+1}y^2 + x^{p-2}y^5&\\\\
&v_2 = z_6 +z _7 = x^{p-2}y^5 + x^{p-5}y^{8}&\\\\ 
&\vdots& \\\\
&v_{(p-4)/3}=z_{p-4} + z_{p-3} = x^9 y^{p-6} + x^6 y^{p-3}&
\\\\
&v_{(p-1)/3}= z_{p-1} + z_{p} = x^6 y^{p-3} + x^3y^{p}&
\end{align*} \text{}\\ and in general, 
\begin{align*}
&v_i = z_{3i} + z_{3i+1}&
\end{align*}\text{}\\
for $1 \leq i \leq (p-1)/3$. By equation \ref{equation:2.2}, we have \\
\begin{equation}\label{equation:2.3}
\sum\limits_{i=1}^{(p-1)/3} v_{i} = x^{p+1}y^2 + y^{p+3} = xy^{p+2} + y^3 \in z \mathfrak{m}^{p+1}.
\end{equation} \text{}\\ By equation \ref{equation:2.3} and the definition of $z_{p+3}$, we have $x^{2}y^{p+1} \in z \mathfrak{m}^{p+1}$. Since $x^{2}y^{p+1}=x^{p+2}y$ by equation \ref{equation:2.1}, we have $x^{p+2}y \in z \mathfrak{m}^{p+1}$. Since $x^{p+3} \in z \mathfrak{m}^{p+1}$, it follows from the definition of the elements $z_i$ that each generator $x^{p+3-i}y^i$ of $\mathfrak{m}^{p+3}$ is contained in $z \mathfrak{m}^{p+1}$. Therefore, $\mathfrak{m}^{p+3} \subseteq z \mathfrak{m}^{p+1}$ and since $\text{ord}_{R}(z)=2$, we have $\mathfrak{m}^{p+3}=z \mathfrak{m}^{p+1}$ and ${(\mathfrak{m}^2)}^{(p+3)/2}=z{(\mathfrak{m}^2)}^{(p+1)/2}$. Therefore, $r_{z}(\mathfrak{m}^2) \leq (p+1)/2$ and \\
\[
p+3 = \text{g}\ell\ell(R) \leq 2(r_{z}(\mathfrak{m}^2)+1) \leq 2\left(\dfrac{p+1}{2}+1\right)=p+3.
\]
\end{proof}
\text{}\\

\begin{proposition}\label{prop:2.9}
Let $p>3$ be a prime such that $3 \mid 2p-1$ and $2$ is a primitive root modulo $p$. Let \\
\[
R=\mathbb{F}_2[[x,y]] \bigg/ \left(xy \left(x^{2p} + y^{2p} \right) \right)
\] \text{}\\ and $\mathfrak{m}=(x,y)R$. Let $z=x^2 + xy + y^2 \in \mathfrak{m}^2 \setminus \mathfrak{m}^3$. Then $z$ is a witness to $\Gll(R)$ and generates a minimal reduction of $\mathfrak{m}^2$. We have \\
\[
2(p+2)=\Gll(R)+1=2 \left(r_{z}(\mathfrak{m}^2) + 1 \right)
\] \text{}\\ and
\[
r_{z}(\mathfrak{m}^2) = p+1.
\] \text{}\\\\
Therefore, $r_{w}(\mathfrak{m}^2)=p+1$ for every principal reduction $w$ of $\mathfrak{m}^2$.
\end{proposition}
\text{}\\
\begin{proof} 
By Proposition \ref{prop:2.6}, $z$ is a witness to $\text{g}\ell\ell(R)$ and $\text{g}\ell\ell(R)=2p+3$. By the proof of Proposition \ref{prop:2.6}, the initial form $z^*$ is $\text{gr}_{\mathfrak{m}}(R)$-regular \cite{Bartels24}[Proposition 3.12]. Therefore, $z$ generates a reduction of $\mathfrak{m}^2$ and \\
\begin{equation*}\label{equation:1}
2p+3=\text{g}\ell\ell(R) \leq 2 \left(r_{z}(\mathfrak{m}^2) + 1 \right).
\end{equation*}\text{}\\
Since $2p+3$ is odd, we have $2p+3<2 \left(r_{z}(\mathfrak{m}^2)+1 \right)$ and \\
\[
2p+4 \leq 2 \left(r_{z}(\mathfrak{m}^2) + 1\right).
\]\text{}\\
We prove that $2p+4=2\left(r_{z}(\mathfrak{m}^2) + 1 \right)$. To this end, we prove the inclusion $\mathfrak{m}^{2p+4} \subseteq z \mathfrak{m}^{2p+2}$. We have 
\[
\mathfrak{m}^{2p+2}=\left(x^{2p+2},\, x^ {2p+1}y, \, \dots, y^{2p+2} \right)R.
\] \text{}\\ The $2p+3$ generators $x^{2p+2-i}y^i$ of $\mathfrak{m}^{2p+2}$ give us the following $2p+3$ generators $z_i$ of $z\mathfrak{m}^{2p+2}$:\\
\begin{align*}
&z_1 = x^{2p+4} + x^{2p+3}y + x^{2p+2}y^2&\\\\
&z_2 = x^{2p+3}y + x^{2p+2}y^2 + x^{2p+1}y^3&\\\\
&\vdots& \\\\
&z_{2p+3}=x^2 y^{2p+2} +xy^{2p+3} + y^{2p+4}&
\end{align*} \text{}\\ and in general, 
\begin{align*}
&z_i=x^{2p+5-i}y^{i-1} + x^{2p+4-i}y^{i} + x^{2p+3-i}y^{i+1}&
\end{align*}\text{}\\ 
for $1 \leq i \leq 2p+3$. Since $x^{2p+1} y= xy^{2p+1}$ in $R$, we have \\ 
\begin{equation}\label{equation:2.5}
x^{2p+3}y=x^3y^{2p+1}
\end{equation} and 
\begin{equation}\label{equation:2.6}
x^{2p+1}y^3 = xy^{2p+3}.
\end{equation} \text{} Define the elements $w_1, \, w_2, \, \dots, \, w_{(2p-1)/3} \in z\mathfrak{m}^{2p+2}$ as follows:\text{}\\
\begin{align*}
&w_1 = z_3 + z_4 = x^{2p+2}y^2 + x^{2p-1}y^5&\\\\
&w_2 = z_6 +z _7 = x^{2p-1}y^5 + x^{2p-4}y^{8}&\\\\
&\vdots& \\\\
&w_{(2p-4)/3}=z_{2p-4} + z_{2p-3} = x^9 y^{2p-5} + x^6 y^{2p-2}&
\\\\
&w_{(2p-1)/3}= z_{2p-1} + z_{2p} = x^6 y^{2p-2} + x^3 y^{2p+1}&
\end{align*} \text{}\\ and in general, 
\begin{align*}
&w_i = z_{3i} + z_{3i+1}&
\end{align*}
\text{}\\ for $1 \leq i \leq (2p-1)/3$. By equation \ref{equation:2.5}, we have \\
\begin{equation}\label{equation:2.7}
\sum\limits_{i=1}^{(2p-1)/3} w_{i} = x^{2p+2}y^2 + x^3y^{2p+1} = x^{2p+2}y^2 + x^{2p+3}y \in z \mathfrak{m}^{2p+2}.
\end{equation} \text{}\\ By equation \ref{equation:2.7} and the definition of $z_1$, it follows that $x^{2p+4} \in z \mathfrak{m}^{2p+2}$. Likewise, define elements $v_1, v_2, \dots, v_{(2p+2)/3}$ as follows: \\
\begin{align*}
&v_1 = z_2 + z_3 = x^{2p+2}y^2 + x^{2p-1}y^5&\\\\
&v_2 = z_5 +z _6 = x^{2p-1}y^5 + x^{2p-4}y^{8}&\\\\ 
&\vdots& \\\\
&v_{(2p-1)/3}=z_{2p-2} + z_{2p-1} = x^7 y^{2p-1} + x^4 y^{2p}&
\\\\
&v_{(2p+2)/3}= z_{2p+1} + z_{2p+2} = x^4 y^{2p} + xy^{2p+3}&
\end{align*} \text{}\\ and in general, 
\begin{align*}
&v_i = z_{3i-1} + z_{3i}&
\end{align*}\text{}\\
for $1 \leq i \leq (2p+2)/3$. By equation \ref{equation:2.6}, we have \\
\begin{equation}\label{equation:2.8}
\sum\limits_{i=1}^{(2p+2)/3} v_{i} = x^{2p+3}y + xy^{2p+3} = x^{2p+3}y + x^{2p+1}y^3 \in z \mathfrak{m}^{2p+2}.
\end{equation} \text{}\\ By equation \ref{equation:2.8} and the definition of $z_{2}$, we have $x^{2p+2}y^{2} \in z \mathfrak{m}^{2p+2}$. It then follows from the definition of the elements $z_i$ that each generator $x^{2p+4-i}y^i$ of $\mathfrak{m}^{2p+4}$ is contained in $z \mathfrak{m}^{2p+2}$. Therefore, $\mathfrak{m}^{2p+4} \subseteq z \mathfrak{m}^{2p+2}$ and since $\text{ord}_{R}(z)=2$, we have $\mathfrak{m}^{2p+4}=z \mathfrak{m}^{2p+2}$ and ${(\mathfrak{m}^2)}^{p+2}=z{(\mathfrak{m}^2)}^{p+1}$. Therefore, $r_{z}(\mathfrak{m}^2) \leq p+1$ and \\\\
\[
2p+4 = \text{g}\ell\ell(R)+1 \leq 2(r_{z}(\mathfrak{m}^2)+1) \leq 2\left(p+2\right)=2p+4.
\]
\end{proof}
\text{}
\begin{remark}
In Proposition \ref{proposition:2.13} we show that, for each prime $p>2$, there is a one-dimensional hypersurface $(R,\mathfrak{m})$ over $\mathbb{F}_p$ such that $\Gll(R)-e(R)=1$ and $2\left(r_{z}(\mathfrak{m}^2) +1 \right)-\text{g}\ell\ell(R)=1$, where $z \in \mathfrak{m}^2$ is an order two witness to $\text{g}\ell\ell(R)$ that generates a reduction of $\mathfrak{m}^2$.
\end{remark}
\text{}\\
\begin{lemma}
Let $p$ be a prime integer. Define the polynomial $f(x,y) \in \mathbb{F}_{p}[x,y]$ by 
\[
f(x,y) = y \left(\prod\limits_{\alpha \in \mathbb{F}_p} (x+\alpha y) \right).
\] Then $f(x,y)=x^p y-xy^p$.
\end{lemma}\text{}\\
\begin{proof}
Let $g(x)=x^p-x \in \mathbb{F}_{p}[x]$. For every element $\alpha \in \mathbb{F}_{p}$, we have $g(\alpha)=0$. Therefore, $(x-\alpha) \mid g(x)$ for all $\alpha \in \mathbb{F}_{p}$. Since the linear polynomials $x-\alpha_1$ and $x-\alpha_2$ are coprime in $\mathbb{F}_p[x]$ for distinct elements $\alpha_1 \neq \alpha_2$ of $\mathbb{F}_p$, we have \\
\[
\left(\prod\limits_{\alpha \in \mathbb{F}_p} (x+\alpha ) \right) \Biggm| g(x)  
\] \text{}\\ and $g(x)=\prod\limits_{\alpha \in \mathbb{F}_p} (x+\alpha)$. Then \\\\
\[
x^p-xy^{p-1}=y^pg(x/y)=\prod\limits_{\alpha \in \mathbb{F}_p} (x+\alpha y) 
\] and 
\[
f(x,y)=x^p y- xy^p.
\]
\end{proof}
\text{}\\
\begin{remark}\label{remark:2.11}
Let $p>2$ be prime. By Euler's criterion, we have the following: \\
\begin{enumerate}
\item[(a)] There is a positive integer $n$ that is a quadratic non-residue modulo $p$.\\
\item[(b)] If $n$ is a quadratic non-residue modulo $p$, then $n^{(p-1)/2}=-1 \,\,\,\,\text{mod} \,\,p$.
\end{enumerate}
\end{remark}
\text{}\\
\begin{proposition}\label{proposition:2.13} 
Let $p>2$ be a prime. Let \\
\[
R=\mathbb{F}_{p}[[x,y]] \biggm/ \left(y \left(\prod\limits_{\alpha \in \mathbb{F}_p} (x+\alpha y) \right) \right) = \mathbb{F}_{p}[[x,y]] \biggm/\left(xy \left(x^{p-1} - y^{p-1} \right) \right)
\] \text{}\\\\ and \,$\mathfrak{m}=(x,y)R$. Let $n$ be a positive integer that is a quadratic non-residue modulo $p$ and $z=x^2-ny^2 \in \mathfrak{m}^2 \setminus \mathfrak{m}^3$. Then $z$ is a witness to $\Gll(R)$ and generates a minimal reduction of $\mathfrak{m}^2$. We have \\
\[
p+3=\Gll(R)+1=2\left(r_{z}(\mathfrak{m}^2) + 1 \right)
\] and \text{}\\ 
\[
r_{z}(\mathfrak{m}^2)=\dfrac{p+1}{2}.
\] \text{}\\\\
Therefore, $r_{w}(\mathfrak{m}^2)=\dfrac{p+1}{2}$ for every principal reduction $w$ of $\mathfrak{m}^2$.
\end{proposition} \text{}\\
\begin{proof}
Since $g(x)=x^2-n$ is a degree two irreducible polynomial over $\mathbb{F}_p$, it follows from Propositions \ref{prop:1.1} and \ref{prop:1.2} and the proof of \cite{Bartels24}[Proposition 3.5] that $z=x^2-ny^2 \in R$ is a witness to $\text{g}\ell\ell(R)$ and $z^*$ is $\text{gr}_{\mathfrak{m}}(R)$-regular. Therefore, $z$ generates a reduction of $\mathfrak{m}^2$ and by \cite{Bartels24}[Proposition 3.5] we have \\
\[
p+2=\text{g}\ell\ell(R) \leq 2 \left(r_{z}(\mathfrak{m}^2) + 1 \right).
\]\text{}\\ Since $p+2$ is odd, we have \\
\[
p+3=\text{g}\ell\ell(R) + 1 \leq 2 \left(r_{z}(\mathfrak{m}^2) + 1 \right).
\] \text{}\\ We show that $\mathfrak{m}^{p+3} \subseteq z \mathfrak{m}^{p+1}$. From the $p+2$ generators $x^{p+1-i}y^i$ of $\mathfrak{m}^{p+1}$, we obtain the following $p+2$ generators $z_i$ of $z \mathfrak{m}^{p+1}$: \text{}\\
\begin{align*}
&z_1 = x^{p+3} -nx^{p+1}y^2&\\\\
&z_2 = x^{p+2}y - nx^{p}y^3&\\\\
&z_3 = x^{p+1}y^2 - nx^{p-1}y^4&\\\\
&\vdots& \\\\
&z_{p+2}=x^2 y^{p+1}-ny^{p+3}&
\end{align*} \text{}\\ and in general, 
\begin{align*}
&z_i=x^{p+4-i}y^{i-1} - n x^{p+2-i}y^{i+1}&
\end{align*}\text{}\\ 
for $1 \leq i \leq p+2$. Define $w \in z \mathfrak{m}^{p+1}$ by \\ 
\[
w=\sum\limits_{i=1}^{(p-1)/2} n^{i-1} z_{2i}.
\] \text{}\\ Then we have 
\[
w= x^{p+2}y -n^{(p-1)/2}x^3 y^p.
\] \text{}\\ Since $x^p y = x y^p$ in $R$, we also have $x^{p+2}y=x^3 y^p$ in $R$. Therefore, \\\\
\[
w=\left(1-n^{(p-1)/2} \right)x^{p+2}y=2 x^{p+2}y
\] \text{}\\ by Remark \ref{remark:2.11} and $x^{p+2}y \in z \mathfrak{m}^{p+1}$. Likewise, define $v \in z \mathfrak{m}^{p+1}$ by \\
\[
v=\sum\limits_{i=1}^{(p-1)/2} n^{i-1}z_{2i+1}.
\] Then 
\[
v = x^{p+1} y^2 - n^{(p-1)/2} x^2 y^{p+1} = x^{p+1} y^2 - n^{(p-1)/2} x^{p+1}y^2 
\] \text{}\\
\[
=\left(1-n^{(p-1)/2} \right) x^{p+1}y^2=2 x^{p+1}y^2 \in z \mathfrak{m}^{p+1}
\] \text{}\\\\ and $x^{p+1}y^2 \in z \mathfrak{m}^{p+1}$. Since we also have $x^{p+2}y \in z \mathfrak{m}^{p+1}$, it follows from the definition of the elements $z_i$ that $\left\{x^{p+3-i}y^i \right\}_{i=1}^{p+4} \subset z\mathfrak{m}^{p+3}$. Therefore, $\mathfrak{m}^{p+3}=z \mathfrak{m}^{p+1}$. So $\left( \mathfrak{m}^2\right)^{(p+3)/2}=z \left(\mathfrak{m}^2 \right)^{(p+1)/2}$ and $r_{z}(\mathfrak{m}^2) \leq \dfrac{p+1}{2}$. We have \\
\[
p+3=\text{g}\ell\ell(R)+1 \leq 2 \left(r_{z}(\mathfrak{m}^2) + 1 \right) \leq 2 \left(\dfrac{p+1}{2} + 1 \right)=p+3.
\] \end{proof}
\text{}
\begin{remark}
By Corollary \ref{corollary:2.6} and Propositions \ref{prop:2.8}, \ref{prop:2.9}, and \ref{proposition:2.13}, we obtain the following theorem.
\end{remark}\text{}\\
\begin{theorem}\label{theorem:2.16}
For every prime $p$, we have the following:\\
\begin{enumerate}
\item[(a)] There exists a one-dimensional hypersurface $(R,\mathfrak{m})$ over $\mathbb{F}_p$ with an order $d \geq 1$ witness $z$ to $\Gll(R)$ such that 
\[
d(r_{z}(\mathfrak{m}^d)+1)-\Gll(R)=0.
\]
\\
\item[(b)] There exists a one-dimensional hypersurface $(R,\mathfrak{m})$ over $\mathbb{F}_p$ with an order $d \geq 1$ witness $z$ to $\Gll(R)$ such that 
\[
d(r_{z}(\mathfrak{m}^d)+1)-\Gll(R)=1.
\]
\end{enumerate}
\end{theorem}
\begin{center}
\section{Obtaining the generalized Loewy length from a reduction}
\end{center} \text{}
In the previous section, we used our knowledge of $\Gll(R)$ for the families of hypersurfaces $\left\{(R,\mathfrak{m}) \right\}$ from \cite{Bartels24}[Section 3] to find the reduction number of $\mathfrak{m}^d$ with respect to witnesses to $\text{g}\ell\ell(R)$. In this section, we show that $z=x^2+xy+y^2$ is a minimal reduction of $\mathfrak{m}^2$ for an additional family of hypersurfaces $\left\{(R,\mathfrak{m}) \right\}$. We use Proposition \ref{prop:1.2} and the inequalities\\
\[
e(R)<\text{g}\ell\ell(R) \leq 2 \left( r_{z}(\mathfrak{m}^2) + 1\right)
\] \text{}\\ to show that 
\[
2(2^{n-1}+1)=\text{g}\ell\ell(R)=2 \left(r_{z}(\mathfrak{m}^2) + 1 \right)
\] \text{}\\ and that $z$ is a witness to $\Gll(R)$ for all hypersurfaces in this family. By \cite{De Stefani16}[Remark 2.3], the initial form $z^*$ is $\Gr_{\mathfrak{m}}(R)$-regular, so we can also find $\text{g}\ell\ell(R)$ and prove that $z$ is a witness to $\text{g}\ell\ell(R)$ with Propositions \ref{prop:1.1} and \ref{prop:1.2}.
\text{}\\\\
\begin{lemma}\label{lemma:3.1}
Let $n \geq 1$. Then \\
\begin{enumerate}
\item[(a)] The binomial coefficient $\binom{2^n}{i}$ is even for all $0 < i< 2^n$.\\
\item[(b)] The binomial coefficient $\binom{2^n-1}{i}$ is odd for all $0 \leq i \leq 2^n -1$.
\end{enumerate}
\end{lemma}
\text{}\\
\begin{proof}
Let $R=\mathbb{F}_2[x,y]$. We show that, for any two polynomials $f, g \in R$ and any positive integer $n \geq 1$, we have 
\[
(f+g)^{2^n}=(f^{2^n}+g^{2^n}).
\] \text{}\\ This identity is clearly true for $n=1$. Suppose it is true for some $n \geq 1$. Then \\\\
\[
(f+g)^{2^{n+1}}={\left((f+g)^{2} \right)}^{2^n} = {(f^2 + g^2)}^{2^n}=f^{2^{n+1}}+g^{2^{n+1}}.
\] \text{}\\\\ Therefore,
\[
x^{2^n} + y^{2^n}=(x+y)^{2^n}=\sum\limits_{i=0}^{2^n} \binom{2^n}{i}x^{2^n-i}y^i
\] \text{}\\\\ for all $n \geq 1$ and $\binom{2^n}{i}=0 \,\, \text{mod} \,\, 2$ for all $0<i<2^n$. Fix $n \geq 1$. To prove part (b) we proceed by induction on $0 \leq i \leq 2^n-1$. For $i=0$, we have $\binom{2^n-1}{0}=1$. Suppose for some $0 \leq i <2^n-1$ the coefficient $\binom{2^n-1}{i}$ is odd. We have \\
\begin{equation}\label{equation:3.1}
\binom{2^n}{i+1} = \binom{2^n-1}{i} + \binom{2^n-1}{i+1}.
\end{equation} \text{}\\ Since $\binom{2^n}{i+1}$ is even by part (a) and $\binom{2^n-1}{i}$ is odd by the induction hypothesis, it follows from equation \ref{equation:3.1} that $\binom{2^n-1}{i+1}$ is odd.
\end{proof}
\text{}
\begin{remark}
By the argument in the proof of \cite{HS06}[Example 8.3.2], the maximal ideal $\mathfrak{m}$ of the ring $(R,\mathfrak{m})$ in Proposition \ref{prop:3.3} is basic. We show that $x^2+xy+y^2 \in R$ is a witness to $\Gll(R)$ that generates a minimal reduction of $\mathfrak{m}^2$. For $n=1$, this follows from \cite{HS97}[Example 3.2 and p.165] and \cite{HS06}[Example 8.3.2].
\end{remark}
\text{}\\
\begin{proposition}\label{prop:3.3}
Let $n \geq 1$. Let \\
\[
R=\mathbb{F}_2[[x,y]] \biggm/ \left( xy(x+y)^{2^n -1}\right)
\] \text{}\\ and $\mathfrak{m}=(x,y)R$. Let $z=x^2+xy+y^2 \in \mathfrak{m}^2 \setminus \mathfrak{m}^3$. Then $z$ is a witness to $\Gll(R)$ and generates a minimal reduction of $\mathfrak{m}^2$. We have\\
\[
2 \left(2^{n-1} +1 \right)=\Gll(R)=2 \left(r_{z}(\mathfrak{m}^2) + 1 \right).
\] \text{}\\ and 
\[
r_{z}(\mathfrak{m}^2)=2^{n-1}.
\] \text{}\\\\
Therefore, $r_{w}(\mathfrak{m}^2)=2^{n-1}$ for every principal reduction $w$ of $\mathfrak{m}^2$.
\end{proposition} 
\text{}\\
\begin{proof}
We show that $\mathfrak{m}^{2^n+2} \subset z \mathfrak{m}^{2^n}$. The $2^n+1$ generators $\left\{x^{2^n-i}y^i \right\}_{i=0}^{2^n}$ of $\mathfrak{m}^{2^n}$ give us the following $2^n+1$ generators $z_i$ of $z\mathfrak{m}^{2^n}$: \\
\begin{align*}
&z_0=x^{{2^n}+2}+x^{{2^n}+1}y+x^{2^n}y^2&
\\\\
&z_1=x^{2^n+1}y + x^{2^n}y^2 + x^{2^n-1}y^3&
\\\\
&\vdots& 
\\\\
&z_{2^n}=x^2 y^{2^{n}} + xy^{{2^{n}}+1} +y^{{2^{n}}+2}&.
\end{align*}
\text{}\\
and in general,
\begin{align*}
&z_i=x^{2^{n}+2-i}y^i+x^{2^{n}+1-i}y^{i+1}+x^{2^{n}-i}y^{i+2}&
\end{align*}
\text{}\\ for $0 \leq i \leq 2^n$. By Lemma \ref{lemma:3.1}, we have \\
\[
(x+y)^{2^n -1}=\sum\limits_{i=0}^{2^n -1} x^{2^n -1-i}y^i
\] \text{}\\
and \\
\begin{equation}\label{equation:3.2}
0 = xy(x+y)^{2^n -1} = \sum\limits_{i=0}^{2^n -1} x^{2^n-i}y^{i+1}.
\end{equation} \text{}\\
There are $2^n$ terms on the right side of equation \ref{equation:3.2}. We either have $2^n=1$ mod $3$ or $2^n=2$ mod $3$. We show that in both cases, $\mathfrak{m}^{2^n+2}=z \mathfrak{m}^{2^n}$. First suppose that $2^n=1$ mod $3$. By equation \ref{equation:3.2}, we have 
\[
x^{2^n+1}y=\sum\limits_{i=1}^{2^n -1} x^{2^n+1-i}y^{i+1} = \sum\limits_{i=1}^{(2^n -1)/3} z_{3i-1} \in z \mathfrak{m}^{2^n}
\] and 
\[
x^{2^n}y^2=\sum\limits_{i=1}^{2^n -1} x^{2^n-i}y^{i+2} = \sum\limits_{i=1}^{(2^n -1)/3} z_{3i} \in z \mathfrak{m}^{2^n}.
\] \text{}\\\\ It then follows from the definition of the elements $z_i$ that $\left\{z_i \right\}_{i=0}^{2^n} \subseteq z \mathfrak{m}^{2^n}$ and $\mathfrak{m}^{2^n+2}=z \mathfrak{m}^{2^n}$. Now suppose that $2^n=2$ mod $3$. By equation \ref{equation:3.2}, we have 
\[
x^{2^n}y + x^{2^n-1}y^2=\sum\limits_{i=2}^{2^n -1} x^{2^n-i}y^{i+1}.
\]\text{}\\ Therefore, 
\[
x^{2^n+1}y + x^{2^n}y^2=\sum\limits_{i=2}^{2^n -1} x^{2^n+1-i}y^{i+1} = \sum\limits_{i=1}^{(2^n-2)/3} z_{3i} \in z \mathfrak{m}^{2^n}
\] and 
\[
x^{2^n}y^2 + x^{2^n-1}y^3=\sum\limits_{i=2}^{2^n -1} x^{2^n-i}y^{i+2} = \sum\limits_{i=1}^{(2^n-2)/3} z_{3i+1} \in z \mathfrak{m}^{2^n}.
\] \text{}\\\\ It then follows from the definition of the elements $z_i$ that $\left\{z_i \right\}_{i=0}^{2^n} \subseteq z \mathfrak{m}^{2^n}$ and $\mathfrak{m}^{2^n+2}=z \mathfrak{m}^{2^n}$. Therefore, 
\[
\left(\mathfrak{m}^2 \right)^{2^{n-1}+1}=z \left(\mathfrak{m}^2 \right)^{2^{n-1}}
\]\text{}\\ and $r_{z}(\mathfrak{m}^2) \leq 2^{n-1}$. Since $R$ has no nonzerodivisors of the form $\alpha x+\beta y$, where $\alpha, \beta \in k$, it follows from Proposition \ref{prop:1.2} that \\\\
\[
2^n+2=e(R)+1 \leq \text{g}\ell\ell(R) \leq 2(r_{z}(\mathfrak{m}^2)+1) \leq 2(2^{n-1}+1)=2^n+2.
\] \text{}\\ So  
\[
2(2^{n-1}+1)=\text{g}\ell\ell(R)=2(r_{z}(\mathfrak{m}^2)+1)
\] \text{}\\ and
\[
r_{z}(\mathfrak{m}^2)=2^{n-1}.
\] 
\end{proof}
\text{}

\section*{Acknowledgements}
R.B. and G.K. would like to thank the Trinity College Faculty Research Committee for financially supporting this research.
\text{} \\

\bibliographystyle{amsplain}
}
\end{document}